\documentclass [11pt]{amsart}
\usepackage {amsmath, amssymb, amscd, mathrsfs, url,pinlabel}
\usepackage[text={6.5in,9in},centering,letterpaper,dvips]{geometry}

\newtheorem {theorem}{Theorem}
\newtheorem {lemma}[theorem]{Lemma}
\newtheorem {proposition}[theorem]{Proposition}

\theoremstyle{remark}
\newtheorem {remark}[theorem]{Remark}

\newtheorem {question}[theorem]{Question}

\newcommand\fiber{\varphi}

\newcommand{\co}{\mskip0.5mu\colon\thinspace}   % Colon for maps.

\def\zz {{\mathbb{Z}}}
\def\bZ {{\mathbb{Z}}}
\def\rr {{\mathbb{R}}}

\def\qq {{\mathbb{Q}}}

\DeclareMathOperator{\rk}{rk}

\def\HFhat{\widehat{HF}}

\title{Non-fibered L-space knots}
\author[Tye Lidman]{Tye Lidman}
\thanks{The first author was supported by a UCLA Dissertation Year Fellowship.}
\address {Department of Mathematics, UT Austin, 1 University Station, Austin, TX 78712}
\email {tlid@math.utexas.edu}

\author[Liam Watson]{Liam Watson}
\thanks{The second author was partially supported by an NSERC Postdoctoral Fellowship.}
\address {School of Mathematics and Statistics, University of Glasgow, 15 University Gardens, Glasgow, UK, G128QW}
\email {liam.watson@glasgow.ac.uk}

\begin{document}

\begin{abstract}
We construct an infinite family of knots in rational homology spheres with irreducible, non-fibered complements, for which every non-longitudinal filling is an L-space.
\end{abstract}

\maketitle

The Heegaard Floer homology of a rational homology three-sphere $Y$ is an abelian group $\HFhat(Y)$ satisfying $\rk\HFhat(Y)\ge|H_1(Y;\bZ)|$ \cite{hfinvariance}. When equality is realized in this bound, $Y$ is called an L-space, and any knot in $Y$ admitting a non-trivial L-space surgery is called an L-space knot \cite{hflens}. A  result of Ghiggini \cite{Ghiggini} and Ni \cite{yihfkfibered} shows that L-space knots in the three-sphere must be fibered. Since manifolds with finite fundamental group provide examples of L-spaces,\footnote{Ozsv\'ath and Szab\'o show that manifolds admitting elliptic geometry are L-spaces  \cite{hflens}; the Geometrization Theorem \cite{KL2008}  implies that three-manifolds with finite fundamental group admit elliptic geometry.} this result implies that a knot $K$ in $S^3$ admitting a finite filling must be fibered. This observation should be compared with other restrictions related to finite fillings such as the Cyclic Surgery Theorem \cite{CGLS} and its extensions \cite{BZ2001}.

The restriction to knots in $S^3$ is not necessary. It is shown in \cite{bbcw} that a primitive knot\footnote{Recall that a knot $K$ is primitive in $Y$ if $[K]\in H_1(Y;\bZ)$ is a generator.} in an irreducible L-space admitting a non-trivial L-space surgery must be fibered. Irreducibility of the complement is required: removing an unknot from an embedded three-ball in any L-space produces a non-fibered manifold with non-trivial L-space fillings. Even in the general setting of knots in rational homology spheres with irreducible complements fibered is not a necessary condition:

\begin{theorem}\label{thm:maintheorem}
There exist infinitely many irreducible, non-fibered knot complements such that all non-longitudinal Dehn fillings are L-spaces. Moreover, these examples arise as knots in manifolds with finite fundamental group.
\end{theorem}

In particular, our examples are non-primitive knots in L-spaces. 

Before turning to the construction, we fix some terminology. Fibrations will always be locally trivial surface bundles over a circle and we say the total space fibers. To avoid confusion, we will refer to Seifert fibrations as Seifert structures; these are foliations of a manifold by circles. The base orbifold is the leaf space of such a foliation, where the (possibly empty) collection of cone points records the multiplicities of the exceptional fibers in the Seifert structure. A circle bundle is a Seifert structure for which there are no exceptional fibers.

Given a three-manifold $M$ with torus boundary, a slope $\alpha$ is a primitive class in $H_1(\partial M;\zz)/\{\pm1\}$. We use $M(\alpha)$ to denote Dehn filling along $\alpha$.  If $\partial M = T_1 \cup T_2$, for tori $T_i$, then we denote $\alpha$-filling on $T_1$ (respectively $T_2$) by $M(\alpha,-)$ (respectively  $M(-,\alpha)$). When $M$ admits a Seifert structure, the slope given by a regular fiber in the boundary is called the fiber slope. For background on Seifert structures and Dehn filling we refer the reader to Boyer \cite{Boyersurvey}.  A key fact is that Dehn filling a Seifert manifold with torus boundary along any slope $\alpha$ other than the fiber slope results in a Seifert manifold with a possible additional singular fiber. The multiplicity of this new fiber is $\Delta(\alpha,\fiber)$, the distance between the slopes $\alpha$ and $\phi$ \cite{Heil1974}.

Finally, for knots in rational homology three-spheres recall that there is a preferred slope given by the rational longitude. This slope is characterized by the property that some number of like-oriented parallel copies in the boundary of the knot complement bounds a properly embedded surface. We will refer to this slope as the longitude. Note that an oriented three-manifold $M$ with torus boundary for which $H_1(M;\qq)\cong\qq$ always arises (non-uniquely) as the complement of a knot in a rational homology three-sphere. 

\section{The twisted $I$-bundle over the Klein bottle}\label{sec:prelims}

Let $N$ denote the twisted $I$-bundle over the Klein bottle. As this orientable three-manifold with torus boundary plays a central role in our construction, we will consider its  construction in depth. 

First consider the group $G$ generated by $f,g\co \rr^3\to\rr^3$ where \begin{align*}f(x,y,z)&=(x+1,y,z)\\g(x,y,z)&=(-x,y+1,-z) \end{align*} and consider the non-compact, orientable three-manifold $N^\circ=\rr^3/G$. Note that the $z$-component of $\rr^3$ gives $N^\circ$ the structure of a line bundle, the zero-section of which is a Klein bottle; this is the unique line bundle over the Klein bottle with orientable total space. By restricting the action of $G$ to $\widetilde{N}=\rr^2\times [-\frac{1}{2},\frac{1}{2}]\subset\rr^3$ we obtain the twisted $I$-bundle over the Klein bottle $N=\widetilde{N}/G$.

\begin{figure}[b]%[ht!]
\labellist
	\pinlabel $a$ at 33 300
	\pinlabel $b$ at 385 300
	\endlabellist
\includegraphics[scale=0.35]{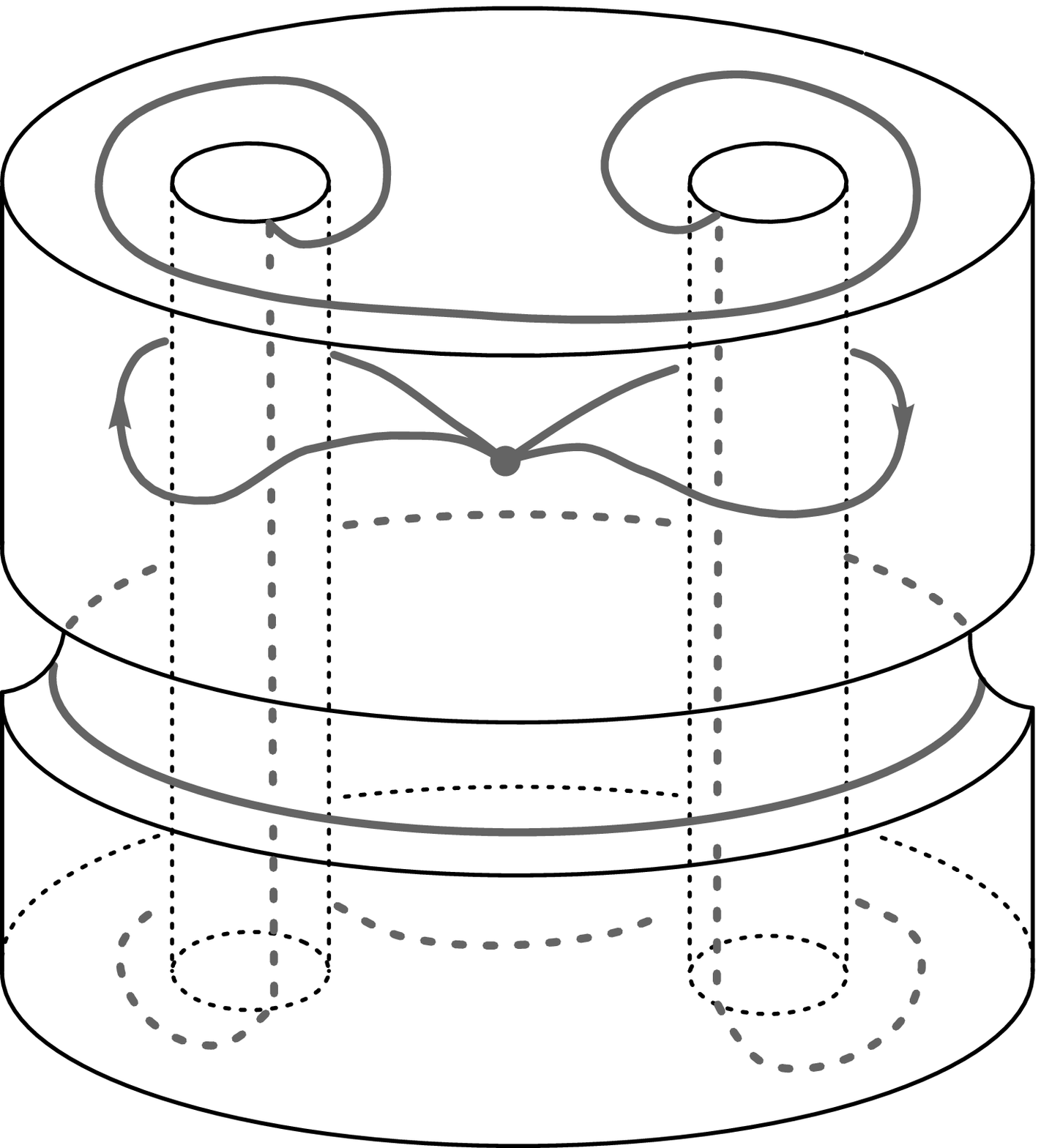}\qquad\quad
\labellist
	\pinlabel $K_0$ at 443 176
	\endlabellist
\includegraphics[scale=0.35]{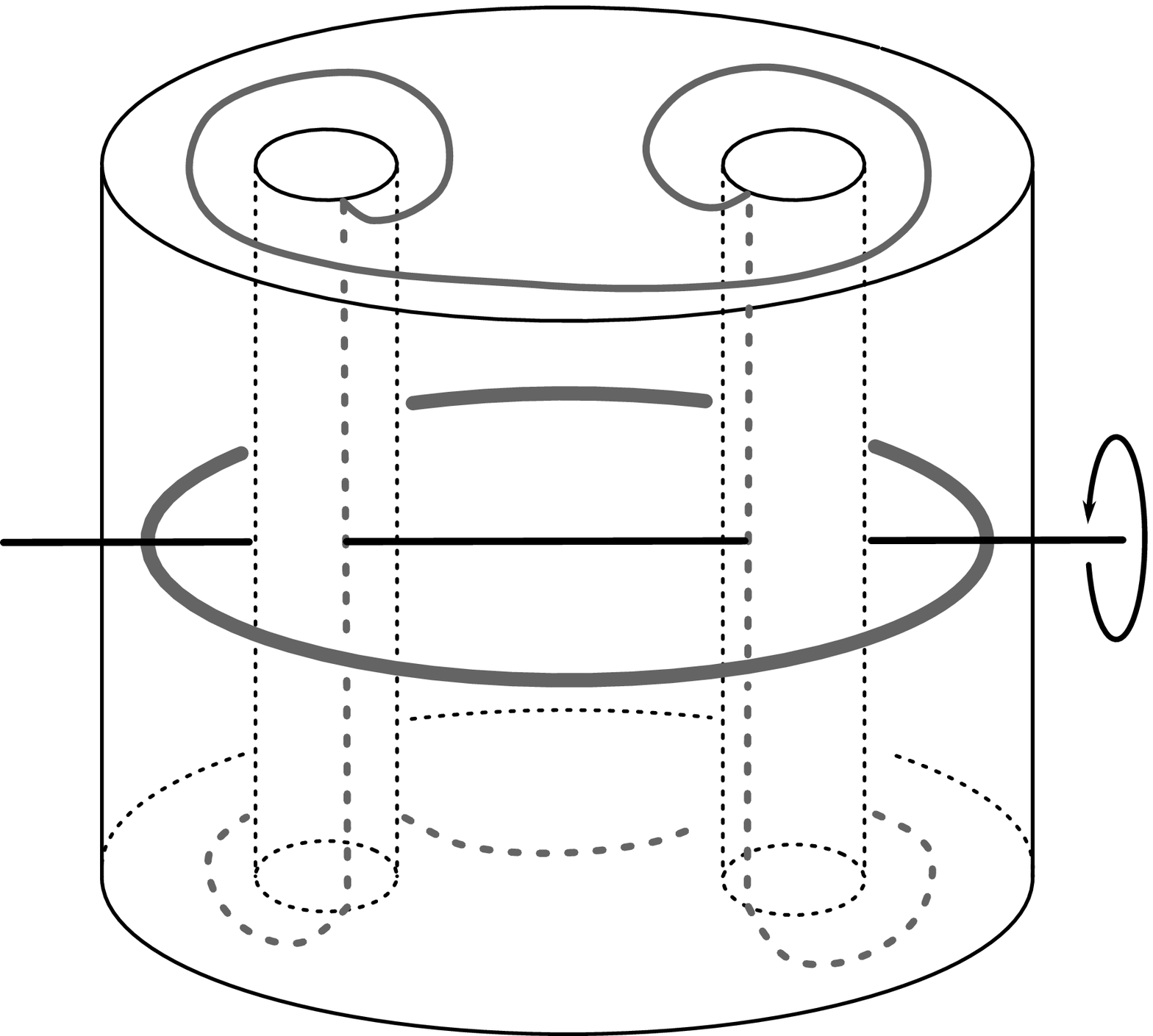}
\caption{Two views of the Heegaard diagram for the twisted $I$-bundle over the Klein bottle $N$. With $a$ and $b$ generating the fundamental group of the genus two handlebody, $N$ is obtained by attaching a handle along a curve in the boundary representing $a^2b^2$ so that $\phi_0\simeq ab$ and $\phi_1\simeq b^2$. On the left, an annulus in the boundary with core representing the element $\phi_0\simeq ab$ may be used to find the fundamental group of $M$, the complement of a regular fiber in the interior of $N$, via HNN extension. On the right, the axis of rotational symmetry shows that the hyperelliptic involution on the handlebody induces a strong inversion on the pair $(N,K_0)$ where $K_0$ is a knot in $N$ isotopic to a regular fiber $\phi_0$ in the interior of $N$.}\label{fig:hnn}\end{figure}

From this description two Seifert structures on $N$ become apparent: the $x$- and $y$-components of $\widetilde{N}$ both determine foliations of $N$ by circles. (This is essentially the observation that the Klein bottle is foliated by circles in two ways.) The leaf space of the foliation described by the $x$-components is a M\"obius strip without cone points. Denote a regular fiber in this Seifert structure by $\phi_0$. The base orbifold of the foliation determined by $y$-components is $D^2(2,2)$, with regular fiber denoted $\phi_1$; this follows readily from a natural Heegaard decomposition which we now describe.

Note that a fundamental domain for $N$ is obtained by taking $[-\frac{1}{2},\frac{1}{2})^2 \times [-\frac{1}{2},\frac{1}{2}] \subset\rr^3$, and removing $D^2\times[-\frac{1}{2},\frac{1}{2}]$ (for some disk of radius less than $\frac{1}{2}$ in the $xy$-plane centered at the origin) gives a genus two handlebody, and hence a Heegaard decomposition for $N$. This Heegaard diagram is described in Figure \ref{fig:hnn}, from which the fundamental group $\pi_1(N) = \langle a,b \mid a^2b^2 \rangle$ may be calculated. Note that since $fgfg^{-1}$ is trivial in the group $G$, the homomorphism determined by $a\mapsto fgf^{-1}$ and $b\mapsto fg^{-1}$ is well-defined and gives an isomorphism $G\cong \langle a,b \mid a^2b^2 \rangle$. Further, by considering a separating disk decomposing the handlebody into solid tori,  it is immediate that $N$ is the union of two solid tori along essential annuli in the boundary. By fixing Seifert structures on each of these solid tori with base orbifolds $D^2(2)$, these annuli are foliated by regular fibers. The identification along these essential annuli therefore extends to a Seifert structure on $N$ with base orbifold $D^2(2,2)$ as claimed.  

Both Seifert structures induce foliations on the torus $\partial N$. Let $\phi_0$  and $\phi_1$ be regular fibers in $\partial N$, and notice that $\Delta(\phi_0,\phi_1)=1$. (These conventions are consistent with \cite[Section 3]{bgw}.) The longitude of $N$ is homotopic to the element $ab$ (this element has order two in the abelianization of $\pi_1(N)$). That is, $\phi_0$ represents the longitude of $N$. Any filling $N(\alpha)$ for which $\alpha\ne\phi_0,\phi_1$ admits a pair of Seifert structures with base orbifolds $\rr P^2(\Delta(\alpha,\phi_0))$ and $S^2(2,2,\Delta(\alpha,\phi_1))$.  We point out that these manifolds always admit elliptic geometry \cite{Scott}.

Now consider a knot $K_0$ in $N$ that is isotopic to a regular fiber $\phi_0$ in the interior of $N$. Define $M$ by removing a neighborhood of $K_0$ from $N$; by construction $M$ inherits a Seifert structure (the base orbifold is a punctured M\"obius band). Now $\partial M = T_1\cup T_2$ where $T_2$ denotes the boundary of a regular neighborhood of $K_0$.

The fundamental group of $M$ is presented by
\[
\pi_1(M)=\langle a,b,t \mid a^2b^2 , [t ,ab] \rangle.
\]
To see this, consult Figure \ref{fig:hnn} and notice that $M$ may be constructed by identifying (disjoint neighborhoods of) each boundary component of the annulus with core $ab$ in $\partial N$.  This gives rise to the HNN extension presented above.  Notice that $M(-,\mu)\cong N$ for any slope on $T_2$ satisfying $\Delta(\mu,\phi_0)=1$.  A preferred choice for $\mu$ is given by a representative of the homotopy class of $t$ in the above presentation.  

\begin{figure}[b]%[ht!]
\labellist \small
	\pinlabel $\phi_1$ at 221 109
	\pinlabel $\longleftarrow$ at 220 90
	\pinlabel $\phi_0$ at 474 109
	\pinlabel $\longrightarrow$ at 475 90
	\endlabellist
\includegraphics[scale=0.35]{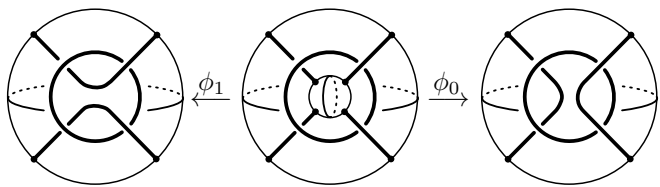}
\caption{The branch set for the manifold $M=M(-,-)$ with branch sets for the fillings $M(\phi_1,-) = N$ and $M(\phi_0, -)$. Notice that $M(\phi_0, -)$ is reducible, containing an $S^2\times S^1$ summand. }\label{fig:montesinos}\end{figure}

A final observation pertains to a natural strong inversion on $(N,K_0)$ that descends to an involution on $M$ with one-dimensional fixed point set. Recall that a strong inversion on $(N,K_0)$ is an orientation preserving involution on $N$ that reverses orientation on $K_0$; such a symmetry is illustrated in Figure \ref{fig:hnn}. The involution on $N$ is induced by the hyperelliptic involution on the genus two handlebody since the attaching curve is fixed (as a set) by this involution. A fundamental domain for this involution is a three-ball, with one dimensional fixed point set. That is, $N$ is the two-fold branched cover of a two-tangle; this is the left-most tangle in Figure \ref{fig:montesinos}. We leave the following step to the reader: the genus two handlebody is the two-fold branched cover of a three-tangle, and attaching the handle closes one of the arcs (the arc meeting the attaching curve) to an unknotted curve in the branch set.   The same construction may be applied to the complement  of $K_0$ in $N$, to see that $M$ is the two-fold branched cover of a tangle in $S^2\times I$. This tangle is shown in Figure \ref{fig:montesinos}.

Towards a proof of Theorem~\ref{thm:maintheorem}, our interest  is in the family of manifolds \[\{M(-,\alpha) \mid \text{for\ any\ slope\ } \alpha \text{\ with\ } \Delta(\alpha,\phi_0)>1\}.\] Notice that each manifold in this set admits a Seifert structure with base orbifold a M\"obius band with a single cone point of order $\Delta(\alpha,\phi_0)$.  Since $M(\phi_1,\alpha)$ admits a Seifert structure with base orbifold $S^2(2,2,n)$  it follows that $M(-,\alpha)$ is the complement of a knot in an elliptic manifold for all $\alpha$.

\section{The proof of Theorem~\ref{thm:maintheorem}}\label{sec:elliptic}

Let $M$ be the complement of $K_0$ in the twisted $I$-bundle over the Klein bottle $N$. We assume all of the notation introduced in the previous section.  

\begin{lemma}\label{lem:fibreablefilling}
Fix a slope $\alpha$ on $T_2$ with $\Delta(\alpha,\phi_0) = p$.  Then
\[
M(\phi_0,\alpha) =
\left\{
\begin{array} {rl}
S^2 \times S^1 \# S^2 \times S^1 &\text{ if }\quad p = 0, \\[2pt]
S^2 \times S^1 \# L(p,q) & \text{ if } \quad p > 1, \\[2pt]
S^2 \times S^1 & \text{ if }\quad p = 1.
\end{array} \right.
\]
\end{lemma}
\begin{proof}
Since
\[
\pi_1(M) \cong \langle a,b,t \mid  a^2b^2, [t,ab] \rangle
\]
and $\phi_0 \simeq ab$, we have that
\begin{align*}
\pi_1(M(\phi_0,-)) & \cong \langle a,b,t \mid a^2b^2, [t,ab] \rangle / \langle \langle ab \rangle \rangle \\ & \cong \langle a,b,t \mid ab  \rangle.
\end{align*}
In other words, $\pi_1(M(\phi_0,-)) \cong \mathbb{Z} * \mathbb{Z}$.  If $\alpha = p \mu + q \phi_0$, then
\begin{align*}
\pi_1(M(\phi_0,\alpha)) & \cong \langle a,b,t \mid ab \rangle / \langle \langle t^p (ab)^q \rangle \rangle \\
& \cong \mathbb{Z} * \mathbb{Z}/p.
\end{align*}
By Whitehead's proof of Kneser's conjecture \cite{whitehead}, $M(\phi_0,\alpha)$ is a connect-sum of closed manifolds $Y_1$ and $Y_2$ with  $\pi_1(Y_1)\cong\mathbb{Z}$ and $\pi_1(Y_2)\cong \mathbb{Z}/p$.  Geometrization now establishes the lemma.
\end{proof}

\begin{remark}
Alternatively, Lemma~\ref{lem:fibreablefilling} follows from considering $M(\phi_0,-)$ as the double branched cover of a tangle as in Figure~\ref{fig:montesinos}.  The unknotted component gives rise to the $S^2 \times S^1$ summand.  Dehn filling corresponds to attaching a rational tangle, which (ignoring the unknotted component) produces a two-bridge link and exhibits the lens space connect-summand.
\end{remark}

\begin{proposition}\label{prop:nofibres}
For any $\alpha$ on $T_2$ with $\Delta(\alpha,\phi_0) >1$, the manifold $M(-,\alpha)$ does not fiber.
\end{proposition}
\begin{proof}
Suppose that $M(-, \alpha)$ fibers.  Since $\phi_0$ is the longitude, this is the only filling that extends the fibration on $M(-,\alpha)$ as any other filling of $M(-,\alpha)$ results in a rational homology sphere.  By Lemma~\ref{lem:fibreablefilling}, $M(\phi_0,\alpha) \cong S^2 \times S^1 \# L(p,q)$ for $p = \Delta(\phi_0,\alpha) \geq 2$.  Since $M(\phi_0,\alpha)$ is fibered and $\pi_2(M(\phi_0,\alpha)) \neq 0$, the fiber surface $F$ must also have $\pi_2(F)\ne0$ by the long exact sequence for a fibration. Hence $F$ must be $S^2$ or $\mathbb{R}P^2$.  However, $\pi_1(M(\phi_0,\alpha))$ is not the fundamental group of such a fibration, since it does not admit a surjective homomorphism onto $\mathbb{Z}$ with finite kernel.
\end{proof}

\begin{proof}[Proof of Theorem~\ref{thm:maintheorem}]
Fix $\alpha$ with $\Delta(\alpha,\phi_0) \geq 2$.  As the fiber slope of the Seifert structure on $M(-,\alpha)$ is the longitude, all non-longitudinal fillings will extend the Seifert structure, yielding a base orbifold $\mathbb{R}P^2$ with two cone points.  By \cite[Proposition 5]{bgw}, such manifolds are always L-spaces.  Proposition~\ref{prop:nofibres} shows that $M(-,\alpha)$ is not fibered.  Furthermore, $M(-,\alpha)$ is irreducible, since the only orientable, reducible Seifert manifolds are $S^2 \times S^1$ and $\mathbb{R}P^3 \# \mathbb{R}P^3$ (and in particular, are closed). Finally, $M(-,\alpha)$ is the complement of a knot in an elliptic manifold as observed in Section \ref{sec:prelims}.
\end{proof}

\begin{remark}
Further examples may be constructed in an analogous way by removing a regular fiber from any manifold which has a Seifert structure with base orbifold $\mathbb{R}P^2$ with any positive number of singular fibers.
It is also possible to construct examples, in a similar manner, admitting Sol geometry. The main observation is that every Sol rational homology sphere is an L-space \cite[Theorem 2]{bgw}. Since every such L-space arises by identifying two twisted $I$-bundles along the boundary tori, one may consider the complement of the knot $K_0$ in one of the twisted $I$-bundles. In this setting, our construction goes through almost verbatim, having noticed that the obvious essential torus must be horizontal to the purported fibration of the exterior of $K_0$.\end{remark}

\begin{question}
All of our examples relied on the presence of an essential annulus, and have non-hyperbolic exterior. Do there exist examples of hyperbolic, non-fibered knots for which every non-longitudinal surgery is an L-space?
\end{question}

\subsection*{Acknowledgments}
The authors thank Ciprian Manolescu and Yi Ni for their comments on and interest in this problem. This paper owes its existence to the Workshop on Topics in Dehn Surgery, held at UT Austin in April 2012. The authors thank the organizers for putting together a great conference.

\bibliographystyle{plain}
\bibliography{references}

\end{document}